\documentclass{amsart}
\usepackage{amsfonts, amsbsy, amsmath, amssymb}
\usepackage{tikz}

\usepackage{pifont}
\newcommand{\cmark}{\ding{51}}%
\newcommand{\xmark}{\ding{55}}%

\newtheorem{thm}{Theorem}[section]
\newtheorem{lem}[thm]{Lemma}
\newtheorem{cor}[thm]{Corollary}
\newtheorem{prop}[thm]{Proposition}
\newtheorem{exmp}[thm]{Example}

\newtheorem{conj}[thm]{Conjecture}

\newtheorem{thm-con}[thm]{Theorem-Conjecture}
\numberwithin{equation}{section}

\newtheorem{taggedtheoremx}{Theorem}
\newenvironment{taggedtheorem}[1]
 {\renewcommand\thetaggedtheoremx{#1}\taggedtheoremx}
 {\endtaggedtheoremx}
 
\newtheorem{taggedlemmax}{Lemma} 
\newenvironment{taggedlemma}[1]
 {\renewcommand\thetaggedlemmax{#1}\taggedlemmax}
 {\endtaggedlemmax} 

\theoremstyle{definition}

\allowdisplaybreaks

\newcommand{\f}{\Bbb F}
\newcommand{\tr}{\text{\rm Tr}}
\newcommand{\wt}{\text{wt}}
\newcommand{\fin}{f_{\text{\rm inv}}}

\begin{document}

\title[On Sum-Free Functions]{On Sum-Free Functions}

\author[A. Ebeling]{Alyssa Ebeling}
\address{School of Mathematical and Physical Sciences, Wisconsin Lutheran College, Milwaukee, WI 53226}
\email{alyssaebeling@wlc.edu}

\author[X. Hou]{Xiang-dong Hou}
\address{Department of Mathematics and Statistics,
University of South Florida, Tampa, FL 33620}
\email{xhou@usf.edu}

\author[A. Rydell]{Ashley Rydell}
\address{Mathematics and Statistics Department,
Sonoma State University,
Rohnert Park, CA 94928}
\email{rydella@sonoma.edu}

\author[S. Zhao]{Shujun Zhao}
\address{Department of Mathematics and Statistics,
University of South Florida, Tampa, FL 33620}
\email{shujunz@usf.edu}



\begin{abstract}
A function from $\Bbb F_{2^n}$ to $\Bbb F_{2^n}$ is said to be {\em $k$th order sum-free} if the sum of its values over each $k$-dimensional $\Bbb F_2$-affine subspace of $\Bbb F_{2^n}$ is nonzero. This notion was recently introduced by C. Carlet as, among other things, a generalization of APN functions. At the center of this new topic is a conjecture about the sum-freedom of the multiplicative inverse function $f_{\text{\rm inv}}(x)=x^{-1}$ (with $0^{-1}$ defined to be $0$). It is known that $f_{\text{\rm inv}}$ is 2nd order (equivalently, $(n-2)$th order) sum-free if and only if $n$ is odd, and it is conjectured that for $3\le k\le n-3$, $f_{\text{\rm inv}}$ is never $k$th order sum-free. The conjecture has been confirmed for even $n$ but remains open for odd $n$. In the present paper, we show that the conjecture holds under each of the following conditions: (1) $n=13$; (2) $3\mid n$; (3) $5\mid n$; (4) the smallest prime divisor $l$ of $n$ satisfies $(l-1)(l+2)\le (n+1)/2$. We also determine the ``right'' $q$-ary generalization of the binary multiplicative inverse function $f_{\text{\rm inv}}$ in the context of sum-freedom. This $q$-ary generalization not only maintains most results for its binary version, but also exhibits some extraordinary phenomena that are not observed in the binary case.
\end{abstract}

\maketitle

\section{Introduction}

In a recent paper \cite{Carlet-0}, Carlet introduced the notion of {\em sum-free} functions. A function $f: \f_{2^n}\to\f_{2^n}$ is said to be {\em$k$th order sum-free} if $\sum_{x\in A}f(x)\ne 0$ for all $k$-dimensional affine subspaces $A$ of $\f_{2^n}$. Unless specified otherwise, subspaces and affine subspaces of $\f_{2^n}$ are over $\f_2$. This definition does not involve the multiplicative structure of $\f_{2^n}$, hence sum-freedom can be defined for functions from $\f_2^n$ to $\f_2^n$. However, the multiplicative structure of $\f_{2^n}$ plays an important role in the constructions and analysis of such functions. As described in \cite{Carlet-0}, the motivation for the notion of sum-free functions was two-fold: First, sum-free functions are a natural generalization of {\em almost perfect nonlinear} (APN) functions; the latter have been extensively studied in cryptography for their resistance against differential attacks \cite{Blondeau-Nyberg-FFA-2015,Carlet-2021,Carlet-Charpin-Zinoviev-DCC-1998,Nyberg-LNCS-1992,Nyberg-LNCS-1994,Nyberg-Knudsen-LNCS-1993}. In particular, APN functions are precisely 2nd order sum-free functions \cite{Hou-DAM-2006}. Second, sum-freedom addresses the vulnerability of block ciphers to integral attacks which exploit the predicability of the sums of values of S-boxes on affine subspaces \cite{Knudsen-Wagner-LNCS-2002,Zaba-Raddum-Henricksen-Dawson-LNCS-2008}. Moreover, from a theoretic point of view, sum-freedom is a notion with rich mathematical content; beneath its simple appearance, there are difficult and enticing questions that require sophisticated mathematical tools. 

Let $\fin$ denote the multiplicative inverse function of $\f_{2^n}$, that is, $\fin(x)=x^{-1}$ for $x\in\f_{2^n}^*$ and $\fin(0)=0$. For $n\ge 2$, we have $\fin(x)=x^{2^n-2}$ for all $x\in\f_{2^n}$. The sum-freedom of $\fin$ has been the focus of several recent papers \cite{Carlet-0,Carlet-pre,Carlet-Hou}. It is well known that $\fin$ is APN (i.e., 2nd order sum-free) if and only if $n$ is odd \cite{Nyberg-LNCS-1994}. Carlet \cite{Carlet-0} determined that $\sum_{x\in A}1/x\ne 0$ for all affine subspaces $A$ of $\f_{2^n}$ not containing $0$. Hence $\fin$ is $k$th order sum-free if and only if $\sum_{0\ne u\in E}1/u\ne 0$ for all $k$-dimensional subspaces $E$ of $\f_{2^n}$. It is trivial that $\fin$ is $1$st order sum-free and it is known from \cite{Carlet-pre} that for $1\le k\le n-1$, $\fin$ is $k$th order sum-free if and only if it is $(n-k)$th order sum-free. Therefore, for $1\le k\le n-1$, $\fin$ is $k$th order sum-free if $k\in\{1,n-1\}$ for even $n$ and $k\in\{1,2,n-2,n-1\}$ for odd $n$. Theoretic results and computer searches, which we will summarize in the next section, strongly suggest that these are the only values of $k$ for which $\fin$ is $k$th order sum-free. This claim has been confirmed in \cite{Carlet-Hou} for even $n$, but for odd $n$ the question remains open. (We remind the reader that the even $n$ case does not represent ``half'' of the question as we will see later that a small prime divisor of $n$ gives an advantage for the treatment of the question.) For $n\ge 2$, define 
\begin{equation}\label{1.1}
\mathcal K_n=\{1\le k\le n-1:\text{$\fin$ is not $k$th order sum-free on $\f_{2^n}$}\}.
\end{equation}
(Note that the meaning of $\mathcal K_n$ here is different from that in \cite{Carlet-Hou}.)
Trivially, $\mathcal K_2, \mathcal K_3,\mathcal K_5$ are empty, and we already know that $\mathcal K_n=\{2,3,\dots n-2\}$ for even $n\ge 4$ and $\mathcal K_n\subset\{3,4,\dots,n-3\}$ for odd $n\ge 7$. Thus the above open question can be formulated as the following
\begin{conj}\label{conj}
For odd $n\ge 7$, $\mathcal K_n=\{3,4,\dots,n-3\}$.
\end{conj}

In the present paper, we obtain several new results concerning the above conjecture. We show that the conjecture holds under any of the following conditions:
\begin{itemize}
\item $n=13$; 
\item $3\mid n$; 
\item $5\mid n$; 
\item the smallest prime divisor $l$ of $n$ satisfies $(l-1)(l+2)\le (n+1)/2$.
\end{itemize}
Because of the last condition, the most difficult cases of Conjecture~\ref{conj} are the scenarios where $n$ is a prime or a product of two primes, not necessarily different, close to each other.
Carlet~\cite{Carlet-pre} proved that $4\in\mathcal K_n$ for all $n\ge 6$. We found a different proof of this fact using the Lang-Weil bound. The new proof is included in the present paper because we believe that the approach can be applied to more general situations.

The notion of sum-freedom can be generalized to functions from $\f_{q^n}$ to $\f_{q^n}$, mutatis mutandis, where $q$ is a prime power. A function $f:\f_{q^n}\to\f_{q^n}$ is said to be {\em $k$th order sum-free} if $\sum_{x\in A}f(x)\ne 0$ for all $k$-dimensional $\f_q$-affine subspaces $A$ of $\f_{q^n}$. In this context, an immediate question is this: What is the $q$-ary version of the binary multiplicative inverse function $\fin$ that maintains the properties of $\fin$? We determine that the ``right'' $q$-ary version is the function $g_{q-1}:\f_{q^n}\to\f_{q^n}$ defined by $g_{q-1}(x)=1/x^{q-1}$ with $1/0$ defined as $0$. The reason for this, as we shall see in Section~4, is not entirely obvious. It turns out that most existing results concerning the sum-freedom of $\fin$ still hold for $g_{q-1}$. However, more interesting things seem to be happening in the $q$-ary case than the binary case. For example, through a computer search, we find that with $q=3,5$ and $n=7$, $g_{q-1}$ on $\f_{q^7}$ is $k$th order sum-free for all $1\le k\le 6$; similar phenomena have not been observed in the binary case. We are not aware of any immediate applications of the functions $g_2(x)=1/x^2$ on $\f_{3^7}$ and $g_4(x)=1/x^4$ on $\f_{5^7}$. Nevertheless, these are extraordinary examples from a mathematical point of view.

The paper is organized as follows: Section~2 provides the preliminaries for the paper. We recall the existing results about binary sum-free functions, especially the multiplicative inverse function $\fin$. We also include a determination of all $(n-1)$th order sum-free functions on $\f_{2^n}$; this question has not been addressed previously. In Section~3, we confirm Conjecture~\ref{conj} under each of the conditions described above and we prove that $4\in\mathcal K_n$ for all $n\ge 6$. In Section~4, we explore $q$-ary sum-free functions with focus on the generalization of the binary multiplicative inverse function $\fin$. It follows from some nontrivial computation that the ``right'' $q$-ary version of $\fin$ is the function $g_{q-1}$ defined by $g_{q-1}(x)=1/x^{q-1}$. We show that most existing results concerning the sum-freedom of $\fin$ still hold for $g_{q-1}$. We also include some limited computer search results on $g_{q-1}$ with $q=3,5$. The appendix contains a proof of the absolute irreducibility of a symmetric homogeneous polynomial of degree 8 in 4 variables over $\f_2$. The proof is rather technical and the result is used in Section~3.

Throughout the paper, a number of tools are used in the proofs and discussions (e.g., the Moore determinant, Carlitz's formula, symmetric polynomials, the Lang-Weil bound). The relations with other areas of mathematics is part of the reason that we find the notion of sum-free functions and the questions that arise from it attractive.


\section{Preliminaries}


\subsection{$k$th order sum-free functions with $k=0,1,2,n-1,n$}\

Let $f:\f_{2^n}\to\f_{2^n}$ be a function. We know that
\begin{itemize}
\item $f$ is $0$th order sum-free if and only if $f(x)\ne 0$ for all $x\in\f_{2^n}$;
\item $f$ is $1$st order sum-free if and only if $f$ is a bijection;
\item $f$ is $2$nd order sum-free if and only if $f$ is APN;
\item $f$ is $n$th order sum-free if and only if $\sum_{x\in\f_{2^n}}f(x)\ne 0$, i.e., $\deg f=2^n-1$ when $f$ is represented by a polynomial modulo $X^{2^n}-X$.
\end{itemize}
The last statement follows from the fact that
\[
\sum_{x\in\f_{2^n}}x^i=\begin{cases}
0&\text{if}\ 0\le i\le 2^n-2,\cr
1&\text{if}\ i=2^n-1.
\end{cases}
\]
There does not appear to be a determination of $(n-1)$th order sum-free functions in the literature. Hence the following theorem is useful for the record.

\begin{thm}\label{n-1order}
Let $f:\f_{2^n}\to \f_{2^n}$ be represented by
\[
f(X)=\sum_{i=0}^{2^n-1}a_iX^i\in\f_{2^n}[X].
\]
Then $f$ is $(n-1)$th order sum-free if and only if $a_{2^n-1}=0$ and
\begin{equation}\label{det-ne0}
\left|
\begin{matrix}
a_{2^n-1-2^0}&a_{2^n-1-2^1}&\cdots &a_{2^n-1-2^{n-1}}\cr
a_{2^n-1-2^1}^2&a_{2^n-1-2^2}^2&\cdots &a_{2^n-1-2^0}^2\cr
\vdots&\vdots&&\vdots\cr
a_{2^n-1-2^{n-1}}^{2^{n-1}}&a_{2^n-1-2^0}^{2^{n-1}}&\cdots &a_{2^n-1-2^{n-2}}^{2^{n-1}}
\end{matrix}
\right|\ne 0.
\end{equation}
\end{thm}

\begin{proof}
When $n=1$, the claim is obviously true, so we assume $n\ge 2$.
\medskip

$(\Leftarrow$) For $0\ne b\in \f_{2^n}$, let
\begin{gather*}
A_b=\{x\in\f_{2^n}:\tr(bx)=0\},\cr
B_b=\{x\in\f_{2^n}:\tr(bx)=1\},
\end{gather*}
where $\tr=\tr_{2^n/2}$. Each hyperplane (affine subspace of co-dimension one) of $\f_{2^n}$ is of the form $A_b$ or $B_b$, depending on whether it contains 0 or not. 
For $0\ne b\in\f_{2^n}$, we have
\begin{align*}
\sum_{x\in B_b}f(x)\,&=\sum_{\substack{x\in\f_{2^n}\cr \tr(bx)=1}}f(x)=\sum_{x\in\f_{2^n}}f(x)\tr(bx)\cr
&=\sum_{x\in\f_{2^n}}\Bigl(\sum_{i=0}^{2^n-1}a_ix^i\Big)\Bigl(\sum_{j=0}^{n-1}b^{2^j}x^{2^j}\Bigr)\cr
&=\sum_{i=0}^{2^n-1}\sum_{j=0}^{n-1}a_ib^{2^j}\sum_{x\in\f_{2^n}}x^{i+2^j}.
\end{align*}
In the above,
\[
\sum_{x\in\f_{2^n}}x^{i+2^j}=
\begin{cases}
-1&\text{if}\ i+2^j=2^n-1,\cr
0&\text{otherwise}.
\end{cases}
\]
Hence
\[
\sum_{x\in B_b}f(x)=\sum_{j=0}^{n-1}a_{2^n-1-2^j}b^{2^j}.
\]
By \cite[Theorem~2.29]{Hou-ams-gsm-2018}, the $\f_2$-linear map $\f_{2^n}\to\f_{2^n}$, $x\mapsto\sum_{j=0}^{n-1}a_{2^n-1-2^j}x^{2^j}$ is bijective if and if condition \eqref{det-ne0} is satisfied. Therefore, $\sum_{x\in B_b}f(x)\ne 0$. Since $a_{2^n-1}=0$, we have $\sum_{x\in\f_{2^n}}f(x)=0$, hence $\sum_{x\in A_b}f(x)=\sum_{x\in B_b}f(x)\ne0$.

\medskip
($\Rightarrow$) Simply trace back the proof of ($\Leftarrow$).
\end{proof}

\subsection{Known results about $\fin$}\

References \cite{Carlet-0,Carlet-pre,Carlet-Hou} contain many results concerning the sum-freedom of the multiplicative inverse function $\fin$ of $\f_{2^n}$. We gather them in their strongest versions for the reader's convenience. We assume that $n\ge 2$ and $1\le k\le n-1$.

\begin{enumerate}
\item If $\gcd(k,n)>1$, then $k\in\mathcal K_n$ (obvious).

\item $\fin$ is 1st order sum-free (obvious).

\item $\fin$ is 2nd order sum-free if and only if $n$ is odd (\cite{Nyberg-LNCS-1994}).

\item $\fin$ is $k$th order sum-free if and only if it is $(n-k)$th order sum-free (\cite[Theorem~3]{Carlet-pre}).

\item If $k,l\in\mathcal K_n$ and $kl<n$, then $k+l\in\mathcal K_n$ (\cite[Theorem~4]{Carlet-pre}).

\item For even $n$, $\fin$ is $k$th order sum-free if and only if $k\in\{1,n-1\}$ (\cite[Theorem~5.2]{Carlet-Hou}).

\item Conjecture~1.1 holds for $n=7,9,11$ (\cite[Table~1]{Carlet-pre}).

\item For $n\ge 6$, $3\in\mathcal K_n$ (\cite[Theorem~6]{Carlet-pre}).

\item For $n\ge 6$, $4\in\mathcal K_n$ (\cite[Theorem~7]{Carlet-pre}).

\item If $X^n-1$ has a factor $X^k+a_{k-1}X^{k-1}+\cdots+a_2X^2+a_0\in\f_2[X]$, then $k\in\mathcal K_n$ (\cite[Corollary~3.2]{Carlet-Hou}).

\item If $n\ge 13k/3+3$, then $k\in\mathcal K_n$ (\cite[Theorem~4.2]{Carlet-Hou}).

\end{enumerate}

\subsection{The Moore determinant}\label{sec2.3}\

The Moore determinant (see \cite{Moore-BAMS-1896}) plays an important role in the study of the sum-freedom of $\fin$ and its $q$-ary version.
Let $X_1,\dots,X_k$ be independent indeterminates. The {\em Moore determinant} in $X_1,\dots,X_k$ over $\f_q$ is defined to be
\[
\Delta(X_1,\dots,X_k)=\left|
\begin{matrix}
X_1&\cdots&X_k\cr
X_1^q&\cdots&X_k^q\cr
\vdots&&\vdots\cr
X_1^{q^{k-1}}&\cdots&X_k^{q^{k-1}}
\end{matrix}\right|\in\f_q[X_1,\dots,X_k].
\]

\begin{lem}[{\cite[Lemma~3.51]{Lidl-Niederreiter-FF-1997}}]\label{Moore}
We have
\[
\Delta(X_1,\dots,X_k)=\prod_{i=1}^k\prod_{a_1,\dots,a_{i-1}\in\f_q}\Bigl(X_i-\sum_{j=1}^{i-1}a_jX_j\Bigr).
\]
\end{lem}

\begin{cor}\label{C2}
A set of $k$ elements
$v_1,\dots,v_k\in\f_{q^n}$ are linearly independent over $\f_q$ if and only if $\Delta(v_1,\dots,v_k)\ne 0$.
\end{cor}

For $0\le i\le k$, let 
\[
\Delta_i(X_1,\dots,X_k)=\left|
\begin{matrix}
X_1&\cdots&X_k\cr
\vdots&&\vdots\cr
X_1^{q^{i-1}}&\cdots&X_k^{q^{i-1}}\cr
X_1^{q^{i+1}}&\cdots&X_k^{q^{i+1}}\cr
\vdots&&\vdots\cr
X_1^{q^{k}}&\cdots&X_k^{q^{k}}
\end{matrix}\right|.
\]
Note that $\Delta_k(X_1,\dots,X_k)=\Delta(X_1,\dots,X_k)$ and $\Delta_0(X_1,\dots,X_k)=\Delta(X_1,\dots,X_k)^q$.

\subsection{Criteria for $\fin$ not to be $k$th order sum-free}\

Recall that $k\in\mathcal K_n$ means that $\fin$ is not $k$th order sum-free. The following criterion was the basis of an approach in \cite{Carlet-Hou}. We follow the notation of \S\ref{sec2.3} with $q=2$.

\begin{thm}[{\cite[Proposition~1]{Carlet-pre}, \cite[\S3.1]{Carlet-Hou}}]\label{criterion}
We have $k\in\mathcal K_n$ if and only if there exist $v_1,\dots,v_k\in\f_{2^n}$ such that $\Delta_1(v_1,\dots,v_k)=0$ but $\Delta(v_1,\dots,v_k)\ne 0$.
\end{thm}

In the above, $\Delta(v_1,\dots,v_k)\ne 0$ if and only if $v_1,\dots,v_k$ are linearly independent over $\f_2$.

There is another criterion by Carlet \cite{Carlet-pre} for $k\in\mathcal K_n$. We will revisit this criterion in Section~3.

\section{New Results on Conjecture~\ref{conj}}

Recall that the set $\mathcal K_n$ is defined in equation \eqref{1.1} and Conjecture~\ref{conj} states that for odd $n\ge 7$, $\mathcal K_n=\{3,4,\dots,n-3\}$. $n=13$ is the first value for which Conjecture~\ref{conj} was not confirmed previously. We now confirm this case with a computer search.

\begin{thm}\label{n=13}
Conjecture~\ref{conj} is true for $n=13$.
\end{thm}

\begin{proof}
By \S2.2 Results (4) and (8), 
we only have to show that $4,5,6\in\mathcal K_{13}$. To this end, for $k=4,5,6$, we search for $v_1,\dots,v_k\in\f_{2^{13}}$, linearly independent over $\f_2$, such that $\Delta_1(v_1,\dots,v_k)=0$. Obviously we may assume $v_1=1$ (since $\fin$ sums to $0$ on $\langle v_1,\dots,v_k\rangle$, the subspace spanned by $v_1,\dots,v_k$, if and only if it sums to $0$ on $\langle av_1,\dots,av_k\rangle$ for any $a\in\f_{2^n}^*$.) We represent $\f_{2^{13}}$ as $\f_2[X]/(f)$, where $f=X^{13}+X^{12}+X^{11}+X^8+1 \in\f_2[X]$ is irreducible (in fact, primitive). The following are examples of the search results:

\medskip
\noindent $k=4$.
\[
\begin{cases}
v_2=X + X^2 + X^5 + X^8 + X^{11},\cr 
v_3=X^4 + X^5 + X^7 + X^8 + X^{11},\cr 
v_4=1 + X^2 + X^5 + X^8 + X^{10} + X^{11}.
\end{cases}
\]

\medskip
\noindent $k=5$.
\[
\begin{cases}
v_2=X^{11},\cr 
v_3=X^{11} + X^{12},\cr 
v_4=X^8 + X^{10} + X^{11},\cr 
v_5=X^3 + X^9.
\end{cases}
\]

\medskip
\noindent $k=6$.
\[
\begin{cases}
v_2=X^{11},\cr 
v_3=X^{11} + X^{12},\cr 
v_4=X^{10},\cr 
v_5=X^6 + X^7 + X^8 + X^9,\cr 
v_6=X + X^2 + X^3 + X^4 + X^8.
\end{cases}
\]
\end{proof}

\noindent{\bf Note.} To prove $6\in \mathcal K_{13}$, we could also use \S2.2 Result (5): $3\in\mathcal K_{13}$, $3\cdot 3<13$ $\Rightarrow$ $3+3=6\in\mathcal K_{13}$.

\medskip
Let $E$ be an $\f_q$-subspace of $\f_{q^n}$ and define
\begin{equation}\label{LE}
L_E(X)=\prod_{u\in E}(X-u)\in\f_{q^n}[X].
\end{equation}
It is well known (\cite[Theorem~3.52]{Lidl-Niederreiter-FF-1997}) that $L_E(X)$ is a $q$-polynomial, i.e., of the form $\sum_{i=0}^da_iX^{q^i}$; hence $L_E:\f_{q^n}\to\f_{q^n}$ is an $\f_q$-linear map with kernel $E$.

\begin{lem}\label{lemUV}
Let $E=\f_{2^l}\subset\f_{2^n}$. Let $U\subset L_E(\f_{2^n})$ be a subspace with $\dim_{\f_2}U=r$ and $\dim_{\f_{2^l}}\f_{2^l}U=s$, where $\f_{2^l}U$ is the $\f_{2^l}$-span of $U$. 
\begin{itemize}
\item[(i)] 
Let $V\subset\f_{2^n}$ be a subspace such that
\begin{equation}\label{dim}
\dim_{\f_2}V=r\quad \text{and}\quad L_E(V)=U.
\end{equation}
(Note that condition \eqref{dim} implies $V\cap E=\{0\}$.) Then
$s\le \dim_{\f_{2^l}}\f_{2^l}V\le s+1$.

\item[(ii)]
We have $\dim_{\f_{2^l}}\f_{2^l}V=s$ for some $V$ satisfying condition \eqref{dim}.
\end{itemize}
\end{lem}

Note that in Lemma~\ref{lemUV}, $s\le r$, which gives Corollary~\ref{c-main1} later.

\begin{proof}[Proof of Lemma~\ref{lemUV}]
(i) We have a short exact sequence of $\f_{2^l}$-subspaces
\[
0\longrightarrow (\f_{2^l}V)\cap E\longrightarrow \f_{2^l}V\xrightarrow{\ L_E\ } L_E(\f_{2^l}V)\longrightarrow 0,
\]
where $L_E(\f_{2^l}V)=\f_{2^l}L_E(V)=\f_{2^l}U$. Hence
\begin{equation}\label{dim-formula}
\dim_{\f_{2^l}}\f_{2^l}V=\dim_{\f_{2^l}}\f_{2^l}U+\dim_{\f_{2^l}}(\f_{2^l}V)\cap E
=\begin{cases}
s&\text{if}\ E\not\subset \f_{2^l}V,\cr 
s+1&\text{if}\ E\subset \f_{2^l}V.
\end{cases}
\end{equation}

\medskip
(ii) Since $L_E:\f_{2^n}\to\f_{2^n}$ is an $\f_{2^l}$-map and $\f_{2^l}U$ is an $\f_{2^l}$-subspace of $L_E(\f_{2^n})$, there exists an $\f_{2^l}$-subspace $W$ of $\f_{2^n}$ with $\dim_{\f_{2^l}}W=s$ such that $L_E(W)=\f_{2^l}U$. Note that $L_E|_W:W\to \f_{2^l}U$ is an $\f_{2^l}$-isomorphism. Let $V=(L_E|_W)^{-1}(U)$. Then $L_E(V)=U$, $\dim_{\f_2}V=\dim_{\f_2}U=r$, and $\dim_{\f_{2^l}}\f_{2^l}V\le \dim_{\f_{2^l}}W=s$.
\end{proof}

\medskip
\noindent{\bf Remark.}
In Lemma~\ref{lemUV} (i), 
\[
\dim_{\f_{2^l}}\f_{2^l}V=s+1\ \Leftrightarrow\ E\cap V=\{0\},\  E\subset\f_{2^l}V\ \text{and}\ U=L_E(V).
\]
Subspaces $V$ satisfying the above conditions do exist. For example, assume $\f_2\subsetneq E=\f_{2^l}\subsetneq \f_{2^n}$. Let $\alpha\in E\setminus\f_2$ and $\beta\in\f_{2^n}\setminus E$, and let $V$ be the $\f_2$-span of $\{\beta,\alpha+\alpha\beta\}$. If $a,b\in\f_2$ are such that $a\beta+b(\alpha+\alpha\beta)\in E\cap V$, then $(a+b\alpha)\beta+b\alpha\in E$, whence $a+b\alpha=0$, i.e., $a=b=0$. Therefore, $E\cap V=\{0\}$. Since $\alpha=(-\alpha)\beta+(\alpha+\alpha\beta)\in\f_{2^l}V$, we have $E=\f_{2^l}\alpha\subset \f_{2^l}V$.

\begin{thm}\label{main1}
Assume $l\ge 2$, $l\mid n$, and that there is an $r$-dimensional $\f_2$-subspace $F$ of $\f_{2^n}$ such that $\sum_{0\ne v\in F}1/v=0$. Let $s=\dim_{\f_{2^l}}\f_{2^l}F$. Then for each $0\le t\le n/l-s$, there is a $(tl+r)$-dimensional $\f_2$-subspace $F_t$ of $\f_{2^n}$ such that $\sum_{0\ne v\in F_t}1/v=0$ and $\dim_{\f_{2^l}}\f_{2^l}F_t=t+s$. In particular, 
\[
\Bigl\{tl+r: 0\le t\le \frac nl-s\Bigr\}\subset\mathcal K_n.
\]
\end{thm}

\begin{proof}
Let $E=\f_{2^l}$. We use induction on $t$. For $t=0$, let $F_0=F$.

Assume that $F_t$ is given, where $t<n/l-s$. Since $\dim_{\f_{2^l}}\f_{2^l}F_t=t+s<n/l$, there exists $0\ne a\in\f_{2^n}$ such that $\tr_{2^n/2^l}(av)=0$ for all $v\in\f_{2^l}F_t$, equivalently, $\tr_{2^n/2^l}(av)=0$ for all $v\in F_t$. Thus $aF_t\subset L_E(\f_{2^n})$. Note that $\dim_{\f_2} aF_t=tl+r$ and $\dim_{\f_{2^l}}\f_{2^l}(aF_t)=\dim_{\f_{2^l}}\f_{2^l}F_t=t+s$. By Lemma~\ref{lemUV} (ii), there is an $\f_2$-subspace $V\subset\f_{2^n}$ such that $L_E(V)=aF_t$, $\dim_{\f_2}V=tl+r$ (hence $V\cap E=\{0\}$) and $\dim_{\f_{2^l}}\f_{2^l}V=t+s$ (hence $(\f_{2^l}V)\cap E=\{0\}$ by equation \eqref{dim-formula}). Let $F_{t+1}=E\oplus V$. Then $\dim_{\f_2}F_{t+1}=l+tl+r=(t+1)l+r$,
\begin{align*}
\dim_{\f_{2^l}}\f_{2^l}F_{t+1}\,&=1+\dim_{\f_{2^l}}\f_{2^l}V\quad\text{(since $E\cap(\f_{2^l}V)=\{0\}$)}\cr
&=t+1+s.
\end{align*}
Moreover, by \cite[Corollary~1]{Carlet-pre}, with $b=\prod_{0\ne u\in E}u$,
\[
\sum_{0\ne u\in F_{t+1}}\frac 1u=\sum_{0\ne u\in E\oplus V}\frac 1u=\sum_{0\ne v\in L_E(V)}\frac bv=\sum_{0\ne v\in F_t}\frac b{av}=\frac ba\sum_{0\ne v\in F_t}\frac 1v=0.
\]
\end{proof}

\begin{cor}\label{c-main1}
Let $l\ge 2$, $l\mid n$. If $r\in\mathcal K_n$, then
\[
\Bigl\{tl+r:0\le t\le \frac nl-r\Bigr\}\subset\mathcal K_n.
\]
\end{cor}

\begin{proof}
Note that in Theorem~\ref{main1}, $s\le r$.
\end{proof}

\begin{exmp}\label{5inK21}\rm
Let $f=X^{21} + X^{19} + 1\in\f_2[X]$, which is irreducible (in fact, primitive) over $\f_2$. Hence $\f_{2^{21}}=\f_2[X]/(f)$. We search for $1,v_2,v_3,v_4\in\f_{2^{21}}$ with $v_2\in\f_{2^3}$ such that $\Delta(1,v_2,v_3,v_4)\ne 0$ but $\Delta_1(1,v_2,v_3,v_4)=0$. The following is an example:
\[
\begin{cases}
v_2=X^{19}+X^{16}+X^{15}+X^{14}+X^{11}+X^9+X^7+X^6+X^4+X^3+X^2+1,\cr
v_3=X^{20}+X^{19}+X^{18}+X^{13}+X^{11}+X^{10}+X^9+X^8+X^4+X^3+1,\cr
v_4=X^{20}+X^{19}+X^{16}+X^{15}+X^{14}+X^{13}+X^{12}+X^{11}+X^{10}+X^8+X^4+X^3+1.
\end{cases}
\]
Let $F$ be the $\f_2$-subspace with basis $\{1,v_2,v_3,v_4\}$. Then $\dim_{\f_{2^3}}\f_{2^3}F\le 3$ since $v_2\in\f_{2^3}$. Now by Theorem~\ref{main1} with $n=21$, $l=3$, $r=4$, $s\le 3$ and $t=4\le n/l-s$, we have $tl+r=4\cdot 3+4=16\in\mathcal K_{21}$. Hence $5=21-16\in\mathcal K_{21}$. 
\end{exmp}

\medskip
\noindent{\bf Note.} In fact, Conjecture~\ref{conj} is true for $n=21$ by \cite[Table~2]{Carlet-Hou}.

\medskip

\begin{thm}\label{main2}
Let $n\ge 3$ be odd and $l$ be the smallest prime divisor of $n$. If 
\begin{equation}\label{ineq}
(l-1)(l+2)\le \frac{n+1}2,
\end{equation}
then Conjecture~\ref{conj} is true for $n$.
\end{thm}

\noindent{\bf Note.} Inequality~\eqref{ineq} is equivalent to $n\ge 2l(l+1)-5$. When $l>5$, this is equivalent to $n\ge 2l(l+1)$ since $l\mid n$.

\begin{proof}[Proof of Theorem~\ref{main2}]
We may assume $n> 13$ since the conjecture is known to be true for $n\le 13$.

\medskip
$1^\circ$ First assume 
\[
l+2\le\frac{3n-9}{13}.
\]
By \cite[Theorem~4.2]{Carlet-Hou}, $3,4,\dots,l+2\in\mathcal K_n$. Then by Theorem~\ref{main1}, for $3\le r\le l+2$,
\[
tl+r\in\mathcal K_n\quad \text{for}\quad 0\le t\le \frac nl-r.
\]
Since 
\[
\Bigl(\frac nl-r\Bigr)l+r=n-(l-1)r\ge n-(l-1)(l+2)\ge n-\frac{n+1}2=\frac{n-1}2,
\]
we see that
\[
\mathcal K_n\supset\Bigl\{tl+r:3\le r\le l+2,\ 0\le t\le \frac nl-r\Bigr\}\supset\Bigl\{3,4,\dots,\frac{n-1}2\Bigr\}.
\]
Since $k\in\mathcal K_n$ if and only if $n-k\in\mathcal K_n$, the conjecture holds in this case.

\medskip
$2^\circ$ Now assume
\begin{equation}\label{l+2}
l+2>\frac{3n-9}{13}.
\end{equation}
Then
\[
\Bigl(\frac{3n-9}{13}-3\Bigr)\frac{3n-9}{13}<\frac{n+1}2,
\]
which implies that $n<27$. The pair $(n,l)$ satisfying inequalities~\eqref{ineq}, where $n$ is odd and $13< n<27$, is $(21,3)$. By \cite[Table~2]{Carlet-Hou}, Conjecture~\ref{conj} is true for $n=21$.
\end{proof}

\begin{cor}\label{3|n}
If $3\mid n$, then Conjecture~\ref{conj} is true for $n$. 
\end{cor}

\begin{proof}
We may assume that $n>13$ since the conjecture is true for $n\le 13$. By Theorem~\ref{main2}, we may assume that
\[
(3-1)(3+2)>\frac{n+1}2,
\]
i.e., $n<19$. Since $3\mid n$, we then have $n=15$. By \cite[Table~2]{Carlet-Hou}, Conjecture~\ref{conj} is true for $n=15$.
\end{proof}

In \cite{Carlet-pre}, Carlet established the following criterion for $k\in\mathcal K_n$ using the indicator function of an $(n-k)$-dimensional subspace of $\f_{2^n}$:

\begin{thm}[{\cite[\S5.2.4]{Carlet-pre}}]\label{Carlet-criterion}
We have $k\in\mathcal K_n$ if and only if there exist $u_1,\dots,u_k\in\f_{2^n}$, linearly independent over $\f_2$, such that
\[
\theta(u_1,\dots,u_k)=0,
\]
where
\begin{equation}\label{theta}
\theta(u_1,\dots,u_k)=\sum_{\substack{b_1,\dots,b_k\in\{-\infty,0,\dots,n-1\}\cr 2^{b_1}+\cdots+2^{b_k}\equiv 1\,(\text{\rm mod}\;2^n-1)}}u_1^{2^{b_1}}\cdots u_k^{2^{b_k}},
\end{equation}
and $2^{-\infty}$ is defined to be $0$.
\end{thm}

This criterion can be reformulated in terms of symmetric polynomials.

\begin{lem}\label{L2.9}
Assume that $2\le l\le 2^n$. Then
\[
\text{\rm wt}_2(1+l(2^n-1))=n+1-\nu_2(l-1),
\]
where $\text{\rm wt}_2$ is the base $2$ weight and $\nu_2$ is the $2$-adic order.
\end{lem}

\begin{proof}
Write
\[
1+l(2^n-1)=2^n-(l-1)+2^n(l-1).
\]
Let the base 2 digits of $l-1$ be
\[
\begin{matrix}
\scriptstyle 0& \scriptstyle 1& \cdots &\scriptstyle i-1& \scriptstyle i& \scriptstyle i+1& \cdots & \scriptstyle n-1\cr
0&0&\cdots&0&1&a_{i+1}&\cdots&a_{n-1}
\end{matrix}\ ,
\]
where $i=\nu_2(l-1)$. Then the base 2 digits of $2^n-(l-1)$ are 
\[
\begin{matrix}
\scriptstyle 0& \scriptstyle 1& \cdots &\scriptstyle i-1& \scriptstyle i& \scriptstyle i+1& \cdots & \scriptstyle n-1\cr
0&0&\cdots&0&1&\overline{a_{i+1}}&\cdots&\overline{a_{n-1}}
\end{matrix}\ ,
\]
where $\overline{a_j}=1-a_j$. Therefore,
\[
\wt_2(1+l(2^n-1))=\wt_2(2^n-(l-1))+\wt_2(l-1)=2+n-1-i=n+1-i.
\]
\end{proof}

\begin{lem}\label{L2.10}
Assume $k\le n$ and $b_1,\dots,b_k\in\{0,1,\dots,n-1\}$ are such that
\[
2^{b_1}+\cdots+2^{b_k}\equiv 1\pmod{2^n-1}.
\]
Then $2^{b_1}+\cdots+2^{b_k}=1$ or $2^n$.
\end{lem}

\begin{proof}
Assume to the contrary that 
\begin{equation}\label{2b1}
2^{b_1}+\cdots+2^{b_k}=1+l(2^n-1)
\end{equation}
for some $l\ge 2$. Then $1+l(2^n-1)\le k 2^{n-1}\le n2^{n-1}$ and hence
\[
l\le \frac{n 2^{n-1}-1}{2^n-1}\le 2^n.
\]
By Lemma~\ref{L2.9},
\begin{equation}\label{wt2}
\wt_2(1+l(2^n-1))=n+1-\nu_2(l-1).
\end{equation}
Since
\[
2^{b_1}+\cdots+2^{b_k}=2^n-(l-1)+2^n(l-1),
\]
the sum at the left hand side has at least $l-1$ carries. Consequently,
\begin{equation}\label{wt2_le}
\wt_2(2^{b_1}+\cdots+2^{b_k})\le k-(l-1).
\end{equation}
Combining equations~\eqref{2b1}, \eqref{wt2} and inequality~\eqref{wt2_le}, we have
\[
k-(l-1)\ge n+1-\nu_2(l-1)\ge n+1-\log_2(l-1).
\]
Hence 
\[
k\ge n+1+l-1-\log_2(l-1)\ge n+1,
\]
which is a contradiction.
\end{proof}

By Lemma~\ref{L2.10}, equation~\eqref{theta} becomes
\begin{equation}\label{theta*}
\theta(u_1,\dots,u_k)=u_1+\cdots+u_k+\sum_{\substack{b_1,\dots,b_k\in\{-\infty,0,\dots,n-1\}\cr 2^{b_1}+\cdots+2^{b_k}=2^n}}u_1^{2^{b_1}}\cdots u_k^{2^{b_k}}.
\end{equation}
Assume $k\ge 2$ and $2^{b_1}+\cdots+2^{b_k}=2^n$. Then $\min\{b_i:b_i\ne-\infty\}\ge n-k+1$ and the equality holds if and only if the multiset $\{b_1,\dots,b_k\}$ equals $\{n-k+1,n-k+1,n-k+2,\dots,n-1\}$. Thus by equation~\eqref{theta*},
\[
\theta(u_1,\dots,u_k)^{2^{k-1}}=u_1^{2^{k-1}}+\cdots+u_k^{2^{k-1}}+\sum_{\substack{b_1,\dots,b_k\in\{-\infty,n-k+1,\dots,n-1\}\cr 2^{b_1}+\cdots+2^{b_k}=2^n}}u_1^{2^{b_1'}}\cdots u_k^{2^{b_k'}},
\]
where $b_i'=b_i-(n-k+1)$. ($-\infty-(n-k+1)=-\infty$.) Note that $b_1,\dots,b_k\in\{-\infty,n-k+1,\dots,n-1\}$ satisfy 
\[
2^{b_1}+\cdots+2^{b_k}=2^n
\]
if and only if $b_1',\dots,b_k'\in\{-\infty,0,\dots,k-2\}$ satisfy
\[
2^{b_1'}+\cdots+2^{b_k'}=2^{-(n-k+1)}2^n=2^{k-1}.
\]
Therefore,
\[
\theta(u_1,\dots,u_k)^{2^{k-1}}=\sum_{l=1}^k\,\sum_{\substack{2^{c_1}\ge\cdots\ge 2^{c_l}\cr 2^{c_1}+\cdots+2^{c_l}=2^{k-1}}}\,\sum_{\substack{(i_1,\dots,i_l)\cr i_1,\dots,i_l\in\{1,\dots,k\}\;\text{distinct}}}u_{i_1}^{2^{c_1}}\cdots u_{i_l}^{2^{c_l}}.
\]
In the above,
\[
\sum_{\substack{(i_1,\dots,i_l)\cr i_1,\dots,i_l\in\{1,\dots,k\}\;\text{distinct}}}u_{i_1}^{2^{c_1}}\cdots u_{i_l}^{2^{c_l}}=m_{(2^{c_1},\dots,2^{c_l})}(u_1,\dots,u_k).
\]
where $m_{(2^{c_1},\dots,2^{c_l})}(X_1,\dots,X_k)$ is the {\em monomial symmetric polynomial} associated to the partition $(2^{c_1},\dots,2^{c_l})$ (\cite[I.2]{Macdonald-1995}). We shall call a partition (of a non-negative integer) a {\em $2$-adic partition} if the parts are powers of $2$. Thus,
\[
\theta(u_1,\dots,u_k)^{2^{k-1}}=\sum_\lambda m_\lambda(u_1,\dots,u_k),
\]
where $\lambda$ runs through all $2$-adic partitions of $2^{k-1}$ with at most $k$ parts, i.e., $\lambda=(2^{c_1},\dots,2^{c_l})\vdash 2^{k-1}$, $l\le k$. Let $\Lambda_k$ denote the set of all $2$-adic partitions of $2^{k-1}$ with at most $k$ parts and define
\[
\Theta_k(X_1,\dots,X_k)=\sum_{\lambda\in\Lambda_k}m_{\lambda}(X_1,\dots,X_k).
\]
With this notation, Theorem~\ref{Carlet-criterion} is reformulated as follows:

\begin{thm}\label{Theta}
We have $k\in\mathcal K_n$ if and only if there are $u_1,\dots,u_k\in\f_{2^n}$ such that $\Delta(u_1,\dots,u_k)\ne0$ but $\Theta_k(u_1,\dots,u_k)=0$.
\end{thm}

Elements of $\Lambda_k$ are of the form $\lambda=(\underbrace{2^{k-1},\dots, 2^{k-1}}_{a_{k-1}},\dots,\underbrace{2^0,\dots,2^0}_{a_0})$, where
\[
\begin{cases}
a_02^0+\cdots+a_{k-1}2^{k-1}=2^{k-1},\cr
a_0+\cdots+a_{k-1}\le k.
\end{cases}
\]
For example, for $\Lambda_4$,
\begin{align*}
(a_0,a_1,a_2,a_3)=\,
&(0,0,0,1),\cr
&(0,0,2,0),\cr
&(0,2,1,0),\cr
&(2,1,1,0),\cr
&(0,4,0,0).
\end{align*}
So,
\[
\Theta_4=m_{(2^3)}+m_{(2^2,2^2)}+m_{(2^2,2,2)}+m_{(2^2,2,1,1)}+m_{(2,2,2,2)}.
\]
The enumeration of all terms in $\Theta_4$ are given in the appendix.

\medskip

The next theorem is an application of Theorem~\ref{Theta}

\begin{thm}\label{k=4}
For $n\ge 6$, $4\in\mathcal K_n$.
\end{thm}

\begin{proof}
For $n\le 15$, the claim is known to be true. (The cases $n=13, 15$ follow from Theorem~\ref{n=13} and Corollary~\ref{3|n}.) Therefore, we assume $n\ge 17$.

Let $q=2^n$ and for $f\in\f_q[X_1,\dots,X_4]$, let 
\[
V_{\f_q^4}(f)=\{(x_1,\dots,x_4)\in\f_q^4:f(x_1,\dots,x_4)=0\}.
\]
By Theorem~\ref{Theta}, it suffices to show that $V_{\f_q^4}(\Theta_4)\not\subset V_{\f_q^4}(\Delta)$. We are able to show that $\Theta_4(X_1,\dots,X_4)$ is irreducible over $\overline\f_2$; see Theorem~\ref{T-A1} in the appendix. Thus by the Lang-Weil bound \cite[Theorem~5.2]{Cafure-Matera-FFA-2006},
\[
|V_{\f_q^4}(\Theta_4)|\ge q^3-(8-1)(8-2)q^{5/2}-5\cdot 8^{13/3}q^2=q^3-42 q^{5/2}-5\cdot 2^{13}q^2,
\]
where $8=\deg\Theta_4$.
By \cite[Lemma~2.2]{Cafure-Matera-FFA-2006},
\[
|V_{\f_q^4}(\Theta_4)\cap V_{\f_q^4}(\Delta)|\le 15^2q^2,
\]
where $15=\deg\Delta$. Hence
\[
|V_{\f_q^4}(\Theta_4)\setminus V_{\f_q^4}(\Delta)|\ge q^3-42 q^{5/2}-5\cdot 2^{13}q^2-15q^2=q^2(q-42 q^{1/2}-41185).
\]
Therefore, it suffices to show that
\begin{equation}\label{suff}
q-42 q^{1/2}-41185>0.
\end{equation}
The larger root of $X^2-42X-41185$ is 
\[
\frac 12(42+\sqrt{42^2+4\cdot 41185})=21+\sqrt{41626}.
\]
Thus, inequality~\eqref{suff} is satisfied when 
\[
q^{1/2}=2^{n/2}>21+\sqrt{41626},
\]
i.e., 
\[
n>2\log_2(21+\sqrt{41626})\approx 15.63.
\]
This completes the proof.
\end{proof}

\begin{thm}\label{5|n}
If $5\mid n$, then Conjecture~\ref{conj} is true for $n$.
\end{thm}

\begin{proof}
We may assume that $n\ge 17$. 
By Theorem~\ref{main2}, we may assume that $(5-1)(5+2)>(n+1)/2$, i.e., $n<55$. By Corollary~\ref{3|n}, we may further assume $3\nmid n$. Therefore, we only have to consider $n=25,35$.

\medskip
When $n=25$, we only have to check that $3,\dots,12\in\mathcal K_{25}$. We already know that $3,4,5\in\mathcal K_{25}$ (\S2.2 Results (1), (6), and Theorem~\ref{k=4}). Using \S2.2 Result (5) repeatedly, we see that $6=3+3,\, 7=3+4,\,\dots,\,11=3+8$ all belong to $\mathcal K_{25}$. By Corollary~\ref{c-main1}, $13=2\cdot 5+3\in\mathcal K_{25}$, whence $12=25-13\in\mathcal K_{25}$.

\medskip
When $n=35$, we only have to check that $3,\dots,17\in\mathcal K_{35}$. As in the above, we have $3,\,4,\,5,\,6(=3+3),\,\dots,\,14(=3+11)\in\mathcal K_{35}$. Also, $15\in\mathcal K_{35}$ since $\gcd(15,35)>1$. By Corollary~\ref{c-main1}, both $18=3\cdot 5+3$ and $19=3\cdot 5+4$ belong to $\mathcal K_{35}$. Therefore, $16=35-19\in\mathcal K_{35}$ and $17=35-18\in\mathcal K_{35}$.
\end{proof}

\section{Generalizations}

Recall that a function $f:\f_{q^n}\to\f_{q^n}$ is said to be {\em $k$th order sum-free} if $\sum_{x\in A}f(x)\ne 0$ for all $k$-dimensional $\f_q$-affine subspaces $A$ of $\f_{q^n}$. In general, compared with their binary counterpart, $q$-ary sum-free functions are more difficult to characterize, construct and analyze. APN functions in odd characteristic, studied in \cite{Dobbertin-Mills-Muller-Pott-Willems-DM-2003, Helleseth-Rong-Sandberg-IEEE-IT-1999}, do not represent the 2nd order $q$-ary sum-free functions. On the other hand, the {\em generalized almost perfect nonlinear} (GAPN) functions in odd characteristic, introduced and explored in \cite{Kuroda-Tsujie-FFA-2017}, appear to be related to the 1st order sum-free functions. However, this connection will not be pursued in the present paper. 

We first take a look of the 1st order and the $(n-1)$th order $q$-ary sum-free functions.

\subsection{The 1st order}\

Let $f:\f_{q^n}\to\f_{q^n}$ be a function represented by $f(X)=\sum_{i=0}^{q^n-1}a_iX^i\in\f_{q^n}[X]$. Then $f$ is 1st order sum-free if and only if $\sum_{x\in\f_q}f(b(x+c))\ne 0$ for all $b\in\f_{q^n}^*$ and $c\in\f_{q^n}$. We have
\[
\sum_{x\in\f_q}f(b(x+c))=\sum_{i=0}^{q^n-1}a_ib^i\sum_{x\in\f_q}(x+c)^i.
\]
There is a polynomial $g_{i,q}\in\f_p[X]$ ($p=\text{char}\,\f_q$) such that 
\begin{equation}\label{giq}
\sum_{x\in\f_q}(x+X)^i=g_{i,q}(X^q-X).
\end{equation}
The polynomial $g_{i,q}$ was introduced and studied in \cite{Hou-FFA-2012} for a different purpose and it satisfies the recursive relation
\[
\begin{cases}
g_{0,q}=\cdots=g_{q-2,q}=0,\cr
g_{q-1,q}=-1,\cr
g_{i,q}=Xg_{i-q,q}+g_{i-q+1,q},\quad i\ge q.
\end{cases}
\]
Therefore,
\[
\sum_{x\in\f_q}f(b(x+c))=\sum_{i=0}^{q^n-1}a_ib^ig_{i,q}(c^q-c).
\]
Consequently, we have the following proposition.

\begin{prop}\label{P4.1}
Let $f(X)=\sum_{i=0}^{q^n-1}a_iX^i\in\f_{q^n}[X]$. Then $f:\f_{q^n}\to\f_{q^n}$ is $1$st order sum-free if and only if 
\[
\sum_{i=0}^{q^n-1}a_ib^ig_{i,q}(d)\ne 0
\]
for all $b\in\f_{q^n}^*$ and $d\in\f_{q^n}$ with $\tr_{q^n/q}(d)=0$.
\end{prop}

In particular, a power function $X^i$ is 1st order sum-free on $\f_{q^n}$ if and only if 
$g_{i,q}(d)\ne 0$ for all $d\in\f_{q^n}$ with $\tr_{q^n/q}(d)=0$. Here, when $d=c^q-c$ ($c\in\f_{q^n}$), 
\[
g_{i,q}(d)=\sum_{x\in\f_q}(x+c)^i=-\sum_{0<j\equiv 0\, \text{(mod $q-1$)}}\binom ij c^{i-j}.
\]

\begin{exmp}\label{E4.2}\rm
Let $q=3$ and $n=2$. We have
\[
g_{i,3}(X)=\begin{cases}
0&\text{if}\ i=0,1,3,\cr
-1&\text{if}\ i=2,4,6,\cr
-X&\text{if}\ i=5,\cr
X&\text{if}\ i=7,\cr
-X^2-1&\text{if}\ i=8.
\end{cases}
\]
For $0\le i\le 8$, $\gcd(g_{i,3}(X^3-X),X^9-X)=1$ if and only if $i=2,4,6$. Hence $X^i$ is 1st order sum-free on $\f_{3^2}$ if and only if $i\in\{2,4,6\}$.
\end{exmp}

\subsection{The $(n-1)$th order}\

Let $f(X)=\sum_{i=0}^{q^n-1}a_iX^i\in\f_{q^n}[X]$, where $n\ge 2$, and let $b\in\f_{q^n}^*$ and $c\in\f_q$. The following computation generalizes that in the proof of Theorem~\ref{n-1order}. We have
\begin{align*}
&\sum_{\tr_{q^n/q}(bx)=c}f(x)\cr
&=\sum_{x\in\f_{q^n}}f(x)\bigl[1-\bigl(\tr_{q^n/q}(bx)-c\bigr)^{q-1}\bigr]\cr
&=\sum_{x\in\f_{q^n}}f(x)-\sum_{x\in\f_{q^n}}f(x)\sum_{j=0}^{q-1}\tr_{q^n/q}(bx)^jc^{q-1-j}\cr
&=-a_{q^n-1}-\sum_{x\in\f_{q^n}}f(x)\sum_{j=0}^{q-1}c^{q-1-j}\bigl((bx)^{q^0}+\cdots+(bx)^{q^{n-1}}\bigr)^j\cr
&=-a_{q^n-1}-\sum_{\substack{x\in\f_{q^n}\cr 0\le j\le q-1}}f(x)c^{q-1-j}\sum_{j_0+\cdots+j_{n-1}=j}\binom{j_0+\cdots+j_{n-1}}{j_0,\dots,j_{n-1}}(bx)^{j_0q^0+\cdots+j_{n-1}q^{n-1}}\cr
&=-a_{q^n-1}-\sum_{j_0+\cdots+j_{n-1}\le q-1}\binom{j_0+\cdots+j_{n-1}}{j_0,\dots,j_{n-1}}c^{q-1-(j_0+\cdots+j_{n-1})}b^{j_0q^0+\cdots+j_{n-1}q^{n-1}}\cr
&\kern 3cm \cdot\Bigl[\sum_{i=0}^{q^n-1}a_i\sum_{x\in\f_{q^n}}x^{i+j_0q^0+\cdots+j_{n-1}q^{n-1}}\Bigr]\cr
&=-a_{q^n-1}+\sum_{j_0+\cdots+j_{n-1}\le q-1}\binom{j_0+\cdots+j_{n-1}}{j_0,\dots,j_{n-1}}a_{q^n-1-(j_0q^0+\cdots+j_{n-1}q^{n-1})}b^{j_0q^0+\cdots+j_{n-1}q^{n-1}}\cr
&\kern 3cm \cdot c^{q-1-(j_0+\cdots+j_{n-1})}.
\end{align*}
When $c=0$,
\begin{align}\label{c=0}
&\sum_{\tr_{q^n/q}(bx)=0}f(x)=\\
&-a_{q^n-1}+\sum_{j_0+\cdots+j_{n-1}=q-1}\binom{q-1}{j_0,\dots,j_{n-1}}a_{q^n-1-(j_0q^0+\cdots+j_{n-1}q^{n-1})}b^{j_0q^0+\cdots+j_{n-1}q^{n-1}}.\nonumber
\end{align}
When $c\ne 0$,
\begin{align}\label{cne0}
&\sum_{\tr_{q^n/q}(bx)=c}f(x)=\\
&\sum_{0<j_0+\cdots+j_{n-1}\le q-1}\binom{j_0+\cdots+j_{n-1}}{j_0,\dots,j_{n-1}}a_{q^n-1-(j_0q^0+\cdots+j_{n-1}q^{n-1})}b^{j_0q^0+\cdots+j_{n-1}q^{n-1}}\cr
&\kern 3cm \cdot c^{q-1-(j_0+\cdots+j_{n-1})}.\nonumber
\end{align}
Therefore, we have the following proposition.
\begin{prop}\label{P4.2}
In the above notation, $f$ is $(n-1)$th order sum-free if and only if 
\begin{equation}\label{cond1}
a_{q^n-1}\ne\sum_{j_0+\cdots+j_{n-1}=q-1}\binom{q-1}{j_0,\dots,j_{n-1}}a_{q^n-1-(j_0q^0+\cdots+j_{n-1}q^{n-1})}b^{j_0q^0+\cdots+j_{n-1}q^{n-1}}
\end{equation}
for all $b\in\f_{q^n}^*$, and 
\begin{align}\label{cond2}
&\sum_{0<j_0+\cdots+j_{n-1}\le q-1}\binom{j_0+\cdots+j_{n-1}}{j_0,\dots,j_{n-1}}a_{q^n-1-(j_0q^0+\cdots+j_{n-1}q^{n-1})}b^{j_0q^0+\cdots+j_{n-1}q^{n-1}}\\
&\kern 3cm \cdot c^{q-1-(j_0+\cdots+j_{n-1})}\ne 0\nonumber
\end{align}
for all $b\in\f_{q^n}^*$ and $c\in\f_q^*$.
\end{prop}

Condition~\eqref{cond1} can always be satisfied since the right side of inequality~\eqref{cond1} takes at most $q^n-1$ values as $b$ runs over $\f_{q^n}^*$. Condition~\eqref{cond2} is difficult to verify in general. However, in special cases, condition \eqref{cond2} can be made more explicit. 

Let $0<t\le q-1$ and assume that for $0<j_0+\cdots+j_{n-1}\le q-1$, 
\[
a_{q^n-1-(j_0q^0+\cdots+j_{n-1}q^{n-1})}=0
\]
unless
\[
\binom{j_0+\cdots+j_{n-1}}{j_0,\dots,j_{n-1}}=0,\ \text{i.e., the sum $j_0+\cdots+j_{n-1}$ has carries in base $p$,}
\]
where $p=\text{char}\,\f_q$, or 
\[
(j_0,\dots,j_{n-1})=(0,\dots,0,\overset it,0\,\dots,0)\quad\text{for some }\ 0\le i\le n-1.
\]
Then the left side of inequality~\eqref{cond2} becomes
\[
c^{q-1-t}\sum_{i=0}^{n-1}a_{q^n-1-tq^i}b^{tq^i}=c^{q-1-t}L_t(b^t),
\]
where 
\[
L_t(X)=\sum_{i=0}^{n-1}a_{q^n-1-tq^i}X^{q^i}
\]
is a $q$-polynomial. Therefore, if $L_t(X)$ is a permutation polynomial of $\f_{q^n}$, then condition \eqref{cond2} is satisfied.

\begin{exmp}\label{E4.4}\rm
Let $q=5$, $n=2$, and $\f_{5^2}=\f_5(\alpha)$, where $\alpha^2=2$. Choose a permutation $q$-polynomial of $\f_{q^2}$, say $X+\alpha X^5$. Let $t=3$ and choose $a_{24-3}$ and $a_{24-3\cdot 5}$ such that 
\[
X+\alpha X^5=L_3(X)=a_{24-3}X+a_{24-3\cdot 5}X^5,
\]
i.e., $a_{21}=1$ and $a_9=\alpha$. For $i\in\{0,\dots,23\}\setminus\{9,21\}$, set $a_i=0$. Then codition~\eqref{cond2} is satisfied. Moreover, the right side of \eqref{cond1} is 0. Therefore, for all $a_{24}\in\f_{5^2}^*$,
\[
f(X)=\alpha X^9+X^{21}+a_{24}X^{24}
\]
is 1st order sum-free on $\f_{5^2}$.
\end{exmp}

For the rest of Section~4, we focus on the $q$-ary version of the binary multiplicative inverse function $\fin$. First, we have to determine the ``right'' $q$-ary generalization of $\fin$. To this end, we need to recall some formulas related to Moore determinant.

\subsection{Formulas about the Moore determinant}\

Recall that the $q$-ary Moore determinant $\Delta(X_1,\dots,X_k)$ and its relatives $\Delta_i(X_1,$ $\dots,X_k)$ are defined in \S\ref{sec2.3}.

\begin{lem}\label{carlitz-formula}\
\begin{itemize}
\item[(i)] We have
\begin{equation}\label{moore-det}
\prod_{a_1,\dots,a_k\in\Bbb F_q}(Y+a_1X_1+\cdots+a_kX_k)=(-1)^{k}\frac{\Delta(Y,X_1,\dots,X_k)}{\Delta(X_1,\dots,X_k)}.
\end{equation}

\item[(ii)] Let $\Delta=\Delta(Y,X_1,\dots,X_k)$ and $\Delta_i=\Delta_i(X_1,\dots,X_k)$.
For $l\ge 0$, we have
\begin{align}\label{Car}
&\sum_{a_1,\dots,a_k\in\Bbb F_q}\frac 1{(Y+a_1X_1+\cdots+a_kX_k)^{l+1}}\\
&=\frac{\Delta_0}\Delta\sum_{l_0q^0+\cdots+l_kq^k=l}\binom{l_0+\cdots+l_k}{l_0,\dots,l_k}\prod_{i=0}^k\Bigl(\frac{(-1)^i\Delta_i)}\Delta\Bigr)^{l_i}.\nonumber
\end{align}
\end{itemize}
\end{lem}

\begin{proof} 
(i) $\Delta(Y,X_1,\dots,X_k)$ is a polynomial of degree $q^k$ in $Y$ with roots $-(a_1X_1+\cdots+a_kX_k)$, $a_1,\dots,a_k\in\Bbb F_q$; $(-1)^{k}\Delta(X_1,\dots,X_k)$ is the leading coefficient.

\medskip
(ii) Differentiating both sides of equation~\eqref{moore-det} with respect to $Y$ gives 
\begin{equation}\label{delta0/delta}
\sum_{a_1,\dots,a_k\in\Bbb F_q}\frac 1{Y+a_1X_1+\cdots+a_kX_k}=\frac{\Delta(X_1,\dots,X_k)^q}{\Delta(Y,X_1,\dots,X_k)}=\frac{\Delta_0}{\Delta}.
\end{equation}
Consider the generating function
\[
G_k(Z)=\sum_{l=1}^\infty\Bigl(\sum_{a_1,\dots,a_k\in\Bbb F_q}\frac 1{(Y+a_1X_1+\cdots+a_kX_k)^l}\Bigr)Z^l.
\]
We have
\begin{align*}
G_k(Z)\,&=\sum_{a_1,\dots,a_k\in\Bbb F_q}\sum_{l=1}^\infty\Bigl(\frac Z{Y+a_1X_1+\cdots+a_kX_k}\Bigr)^l\cr
&=\sum_{a_1,\dots,a_k\in\Bbb F_q}\frac{\displaystyle \frac Z{Y+a_1X_1+\cdots+a_kX_k}}{\displaystyle 1-\frac Z{Y+a_1X_1+\cdots+a_kX_k}}\cr
&=Z\sum_{a_1,\dots,a_k\in\Bbb F_q}\frac 1{Y-Z+a_1X_1+\cdots+a_kX_k}\cr
&=Z\,\frac{\Delta_0}{\Delta(Y-Z,X_1,\dots,X_k)}\kern 7em\text{(by eq. \eqref{delta0/delta})}\cr
&=Z\,\frac{\Delta_0}{\Delta-\Delta(Z,X_1,\dots,X_k)}\cr
&=Z\,\frac{\Delta_0}{\Delta}\cdot\frac 1{\displaystyle 1-\frac{\Delta(Z,X_1,\dots,X_k)}{\Delta}}\cr
&=Z\,\frac{\Delta_0}{\Delta}\sum_{l=0}^\infty\frac{\Delta(Z,X_1,\dots,X_k)^l}{\Delta^l}.
\end{align*}
In the above,
\begin{align*}
&\Delta(Z,X_1,\dots,X_k)^l=\Bigl(\sum_{i=0}^k(-1)^i\Delta_iZ^{q^i}\Bigr)^l\cr
&=\sum_{l_0+\cdots+l_k=l}\binom l{l_0,\dots,l_k}\Bigl(\prod_{i=0}^k((-1)^i\Delta_i)^{l_i}\Bigr)Z^{l_0q^0+\cdots+l_kq^k}.
\end{align*}
Therefore
\[
G_k(Z)=Z\,\frac{\Delta_0}\Delta\sum_{l=0}^\infty\;\sum_{l_0q^0+\cdots+l_kq^k=l}\binom{l_0+\cdots+l_k}{l_0,\dots,l_k}\Bigl(\prod_{i=0}^k\Bigl(\frac{(-1)^i\Delta_i)}\Delta\Bigr)^{l_i}\Bigr)Z^l.
\]
Hence
\begin{align*}
&\sum_{a_1,\dots,a_k\in\Bbb F_q}\frac 1{(Y+a_1X_1+\cdots+a_kX_k)^{l+1}}\cr
&=\frac{\Delta_0}\Delta\sum_{l_0q^0+\cdots+l_kq^k=l}\binom{l_0+\cdots+l_k}{l_0,\dots,l_k}\prod_{i=0}^k\Bigl(\frac{(-1)^i\Delta_i)}\Delta\Bigr)^{l_i}.
\end{align*}
\end{proof}

\medskip
\noindent{\bf Remark.} Equation \eqref{Car} leads to a formula by Carlitz on the power sum of reciprocals of polynomials over finite fields \cite[Theorem~9.2]{Carlitz-DMJ-1935}.
 
\subsection{What is the natural $q$-ary generalization of the binary multiplicative inverse function $\fin$?}\

For $1\le s\le q-1$, let $g_s:\f_{q^n}\to\f_{q^n}$ be defined by 
\[
g_s(x)=\begin{cases}
1/x^s&\text{if}\ x\ne 0,\cr
0&\text{if}\ x=0.
\end{cases}
\]
(If $q^n>s+1$, then $g_s(x)=x^{q^n-1-s}$ for all $x\in\f_{q^n}$.) We remind the reader that $g_s$ here is different from the polynomial $g_{i,q}$ in equation~\eqref{giq}. By equation~\eqref{Car}, 
\begin{equation}\label{s-power-sum}
\sum_{a_1,\dots,a_k\in\f_q}\frac 1{(Y+a_1X_1+\cdots+a_kX_k)^s}=\Bigl(\frac{\Delta_0}\Delta\Bigr)^s.
\end{equation}
Hence, if $E$ is a $k$-dimensional $\f_q$-subspace of $\f_{q^n}$ and $y\in\f_{q^n}\setminus E$, then
\[
\sum_{u\in E}\frac 1{(y+u)^s}\ne 0.
\]
We have
\[
\sum_{0\ne(a_1,\dots,a_k)\in\f_q^k}\frac 1{(Y+a_1X_1+\cdots+a_kX_k)^s}=\Bigl(\frac{\Delta_0}\Delta\Bigr)^s-\frac 1{Y^s}.
\]
Note that 
\[
\Delta=\sum_{i=0}^k(-1)^iY^{q^i}\Delta_i=Y\Delta_0\Bigl(1-Y^{q-1}\frac{\Delta_1}{\Delta_0}+\text{terms of higher degree in $Y$}\Bigr).
\]
Hence,
\[
\Delta^{-s}=Y^{-s}\Delta_0^{-s}\Bigl(1+sY^{q-1}\frac{\Delta_1}{\Delta_0}+\cdots\Bigr).
\]
Therefore,
\begin{align*}
\sum_{0\ne(a_1,\dots,a_k)\in\f_q^k}\frac 1{(Y+a_1X_1+\cdots+a_kX_k)^s}
\,&= Y^{-s}\Bigl(1+sY^{q-1}\frac{\Delta_1}{\Delta_0}+\cdots\Bigr)-\frac 1{Y^s}\cr
&\equiv sY^{q-1-s}\frac{\Delta_1}{\Delta_0}\pmod{Y^{q-s}}.
\end{align*}
Setting $Y=0$ gives
\[
\sum_{0\ne(a_1,\dots,a_k)\in\f_q^k}\frac 1{(a_1X_1+\cdots+a_kX_k)^s}=
\begin{cases}
-\displaystyle\frac{\Delta_1}{\Delta_0}&\text{if}\ s=q-1,\vspace{0.2em}\cr
0&\text{if}\ 1\le s<q-1.
\end{cases}
\]
Let $E$ be a $k$-dimensional $\f_q$-subspace of $\f_{q^n}$ with basis $v_1,\dots,v_k$. Then
\[
\sum_{0\ne u\in E}\frac 1{u^s}=
\begin{cases}
-\displaystyle\frac{\Delta_1(v_1,\dots,v_k)}{\Delta(v_1,\dots,v_k)^q}&\text{if}\ s=q-1,\vspace{0.2em}\cr
0&\text{if}\ 1\le s<q-1.
\end{cases}
\]
When $s<q-1$, $g_s$ is never $k$th order sum-free. Therefore, the ``right'' $q$-ary generalization of $\fin$ is $g_{q-1}$. From the above, we know that $g_{q-1}$ is not $k$th order sum-free if and only if there exist $v_1,\dots,v_k\in\f_{q^n}$ such that $\Delta(v_1,\dots,v_k)\ne 0$ but $\Delta_1(v_1,\dots,v_k)=0$. 
We state this fact as a theorem since it is the foundation of some subsequent results.

\begin{thm}[$q$-ary version of Theorem~\ref{criterion}]\label{q-criterion}
The function $g_{q-1}$ is not $k$th order sum-free if and only if there exist $v_1,\dots,v_k\in\f_{q^n}$ such that $\Delta(v_1,\dots,v_k)\ne 0$ but $\Delta_1(v_1,\dots,v_k)=0$.
\end{thm}

\noindent{\bf Note.} In Theorem~\ref{q-criterion}, we may take $v_1=1$.

\subsection{Generalized results for $g_{q-1}$}\

We shall see that most results about $\fin$ on $\f_{2^n}$ can be generalized to $g_{q-1}$ on $\f_{q^n}$.

\begin{prop}\label{1-2-order}\
\begin{itemize}
\item[(i)] The function $g_{q-1}$ is $1$st order sum-free.

\item[(ii)] Let $n\ge 2$. Then $g_{q-1}$ is $2$nd order sum-free if and only if $n$ is odd.
\end{itemize}
\end{prop}

\begin{proof}
(i) Immediate from Theorem~\ref{q-criterion}.

\medskip
(ii) Let $v_2\in\f_{q^n}$. Then 
\[
\Delta(1,v_2)\ne 0\ \Leftrightarrow\ \left|\begin{matrix} 1&v_2\cr 1&v_2^q\end{matrix}\right|\ne 0\ \Leftrightarrow\ v_2\notin\f_q,
\]
and 
\[
\Delta_1(1,v_2)=0\ \Leftrightarrow\ \left|\begin{matrix} 1&v_2\cr 1&v_2^{q^2}\end{matrix}\right|= 0\ \Leftrightarrow\ v_2^{q^2}=v_2\ \Leftrightarrow\  v_2\in\f_{q^2}.
\]
Hence,
\begin{align*}
&\text{there exists $v_2\in\f_{q^2}$ such that $\Delta(1,v_2)\ne 0$ but $\Delta_1(1,v_2)=0$}\cr
&\Leftrightarrow \f_{q^2}\subset\f_{q^n} \Leftrightarrow \text{$n$ is even}.
\end{align*}
\end{proof}

For $n\ge 2$, define
\[
\mathcal K_{n,q}=\{1\le k\le n-1:\text{$g_{q-1}$ is not $k$th order sum-free on $\f_{q^n}$}\}.
\]
When $q=2$, $\mathcal K_{n,2}=\mathcal K_n$, which is defined in equation~\eqref{1.1}.

\begin{prop}\label{gcd}
If $\gcd(k,n)>1$, then $k\in\mathcal K_{n,q}$.
\end{prop}

\begin{proof}
Let $d=\gcd(k,n)$ and let $E$ be a $k/d$-dimensional $\f_{q^d}$-subspace of $\f_{q^n}$. Then $\dim_{\f_q}E=k$. Since $d>1$, there exists $a\in\f_{q^d}\setminus\f_q$. Since $aE=E$, we have
\[
\sum_{0\ne u\in E}\frac 1{u^{q-1}}=\sum_{0\ne u\in E}\frac 1{(au)^{q-1}}=a^{1-q}\sum_{0\ne u\in E}\frac 1{u^{q-1}},
\]
where $a^{1-q}\ne 1$. Hence 
\[
\sum_{0\ne u\in E}\frac 1{u^{q-1}}=0.
\]
\end{proof}

\begin{lem}\label{E+F-q-ary}
Let $E$ and $F$ be subspaces of $\f_{q^n}$ such that $E\cap F=\{0\}$. Then
\begin{equation}\label{E+F-q}
\sum_{0\ne u\in E\oplus F}\frac 1{u^{q-1}}=\sum_{0\ne u\in E}\frac 1{u^{q-1}}+\Delta(v_1,\dots,v_k)^{(q-1)^2}\sum_{0\ne v\in L_E(F)}\frac 1{v^{q-1}},
\end{equation}
where $v_1,\dots,v_k$ is a basis of $E$ and $L_E(X)=\prod_{u\in E}(X+u)$.
\end{lem}

\begin{proof}
We have
\begin{equation}\label{E+F-eq1}
\sum_{0\ne u\in E\oplus F}\frac 1{u^{q-1}}=\sum_{0\ne u\in E}\frac 1{u^{q-1}}+\sum_{0\ne w\in F}\,\sum_{u\in E}\frac 1{(w+u)^{q-1}}.
\end{equation}
By equation~\eqref{s-power-sum},
\[
\sum_{u\in\langle v_1,\dots,v_k\rangle}\frac 1{(Y+u)^{q-1}}=\Bigl(\frac{\Delta_0(v_1,\dots,v_k)}{\Delta(Y,v_1,\dots,v_k)}\Bigr)^{q-1}.
\]
In the above, $\Delta_0(v_1,\dots,v_k)=\Delta(v_1,\dots,v_k)^q$, and by equation~\eqref{moore-det},
\[
\Delta(Y,v_1,\dots,v_k)=(-1)^k L_E(Y)\Delta(v_1,\dots,v_k).
\]
Hence
\begin{equation}\label{E+F-eq2}
\sum_{u\in\langle v_1,\dots,v_k\rangle}\frac 1{(Y+u)^{q-1}}=\Bigl(\frac{\Delta(v_1,\dots,v_k)^q}{(-1)^k L_E(Y)\Delta(v_1,\dots,v_k)}\Bigr)^{q-1}=\frac{\Delta(v_1,\dots,v_k)^{(q-1)^2}}{L_E(Y)^{q-1}}.
\end{equation}
Using equation~\eqref{E+F-eq2} in \eqref{E+F-eq1} gives
\[
\sum_{0\ne u\in E\oplus F}\frac 1{u^{q-1}}=\sum_{0\ne u\in E}\frac 1{u^{q-1}}+\Delta(v_1,\dots,v_k)^{(q-1)^2}\sum_{0\ne w\in F}\frac 1{L_E(w)^{q-1}}.
\]
Since $\ker(L_E:\f_{q^n}\to\f_{q^n})=E$, we have $\ker(L_E)\cap F=E\cap F=\{0\}$. Hence the restriction $L_E|_F:F\to L_E(F)$ is an isomorphism. Therefore,
\[
\sum_{0\ne w\in F}\frac1{L_E(w)^{q-1}}=\sum_{0\ne v\in L_E(F)}\frac 1{v^{q-1}}.
\]
This completes the proof of the lemma.
\end{proof}

\begin{thm}[$q$-ary version of {\cite[Theorem~3]{Carlet-pre}}]\label{q-duality}
Assume $n\ge 2$. If $E$ and $F$ are subspaces of $\f_{q^n}$ such that $\f_{q^n}=E\oplus F$, then
\[
\sum_{0\ne u\in E}\frac 1{u^{q-1}}\ne 0\ \Leftrightarrow\ \sum_{0\ne v\in L_E(F)}\frac 1{v^{q-1}}\ne 0.
\]
In particular, $k\in\mathcal K_{n,q}$ if and only if $n-k\in\mathcal K_{n,q}$.
\end{thm}

\begin{proof}
In equation~\eqref{E+F-q}, we have
\[
\sum_{0\ne u\in E\oplus F}\frac 1{u^{q-1}}=\sum_{0\ne u\in\f_{2^n}}\frac 1{u^{q-1}}=0
\]
since $n\ge 2$. The conclusions follow immediately.
\end{proof}

For $\f_q$-subspaces $E$ and $F$ of $\f_{q^n}$, let $EF$ denote the $\f_q$-span of $\{xy:x\in E,\ y\in F\}$ and let $F^\bot=\{x\in\f_{q^n}:\tr_{q^n/q}(xy)=0\ \text{for all}\ y\in F\}$. 

\begin{thm}[$q$-ary version of \S2.2 Result (5)]\label{k+l}
If $k,l\in\mathcal K_{n,q}$ and $kl<n$, then $k+l\in\mathcal K_{n,q}$.
\end{thm}

\begin{proof}
Let $E$ and $F$ be $\f_q$-subspaces of $\f_{q^n}$ such that $\dim_{\f_q}E=k$, $\dim_{\f_q}F=l$, and $\sum_{0\ne u\in E}1/u^{q-1}=0$, $\sum_{0\ne v\in F}1/v^{q-1}=0$. Since
\[
\dim_{\f_q}\bigl(F\cdot L_E(\f_{q^n})^\bot\bigr)\le\dim_{\f_q}F\cdot\dim_{\f_q}E=lk<n,
\]
there exists $0\ne a\in\f_{q^n}$ such that $\tr_{q^n/q}\bigl(aF\cdot L_E(\f_{q^n})^\bot\bigr)=0$, i.e., $aF\subset L_E(\f_{q^n})$. Since $L_E:\f_{q^n}\to\f_{q^n}$ induces an isomorphism $\f_{q^n}/E\to L_E(\f_{q^n})$, there exists an $\f_q$-subspace $V$ of $\f_{q^n}$ such that 
$V\cap E=\{0\}$, $\dim_{\f_q}V=l$ and $L_E(V)=aF$. Let $b=\Delta(v_1,\dots,v_k)^{(q-1)^2}$, where $v_1,\dots,v_k$ is a basis of $E$ over $\f_q$. Now by Lemma~\ref{E+F-q-ary},
\begin{align*}
\sum_{0\ne u\in E\oplus V}\frac 1{u^{q-1}}\,&=\sum_{0\ne u\in E}\frac 1{u^{q-1}}+b\sum_{0\ne v\in L_E(V)}\frac 1{v^{q-1}}\cr
&=b\sum_{0\ne v\in aF}\frac 1{v^{q-1}}
=\frac b{a^{q-1}}\sum_{0\ne v\in F}\frac 1{v^{q-1}}= 0.
\end{align*}
Hence $k+l\in\mathcal K_{n,q}$.
\end{proof}

\begin{thm}[$q$-ary version of {\cite[Theorem~3.1]{Carlet-Hou}}]\label{factor-xn-1}
The following two statements are equivalent:
\begin{itemize}
\item[(i)] There exists $x\in\f_{q^n}$ such that $\Delta(x,x^q,x^{q^2},\dots,x^{q^{k-1}})\ne0$ and\\ $\Delta_1(x,x^q,x^{q^2},\dots,x^{q^{k-1}})=0$.

\item[(ii)] $X^n-1$ has a factor $X^k+a_{k-1}X^{k-1}+\cdots+a_2X^2+a_0\in\f_q[X]$.
\end{itemize}
\end{thm}

\begin{proof}
Let $\sigma$ denote the Frobenius automorphism  of $\f_{q^n}$ over $\f_q$. Note that 
\begin{align*}
&\Delta(x,x^q,x^{q^2},\dots,x^{q^{k-1}})\ne0\cr
\Leftrightarrow\ &x,\sigma(x),\dots,\sigma^{k-1}(x)\ \text{are linearly independent over}\ \f_2,
\end{align*}
and 
\begin{align*}
&\Delta_1(x,x^q,x^{q^2},\dots,x^{q^{k-1}})=0\cr
\Leftrightarrow\ &x,\sigma^2(x),\dots,\sigma^{k}(x)\ \text{are linearly dependent over}\ \f_2.
\end{align*}

\medskip
(ii) $\Rightarrow$ (i). Let $\alpha\in\f_{q^n}$ be a normal element over $\f_q$. Write $X^n-1=fg$, where $f=X^k+a_{k-1}X^{k-1}+\cdots+a_2X^2+a_0$. Let $x=(g(\sigma))(\alpha)$. For each $0\ne h\in\f_q[X]$ with $\deg h<k$, $hg\not\equiv 0\pmod{X^n-1}$, whence $(h(\sigma))(x)=((hg)(\sigma))(\alpha)\ne 0$. Hence $x,\sigma(x),\dots,\sigma^{k-1}(x)$ are linearly independent over $\f_q$. 

Since $(f(\sigma))(x)=((fg)(\sigma))(\alpha)=0$, the elements $x,\sigma^2(x),\dots,\sigma^k(x)$ are linearly dependent over $\f_q$.

\medskip
(i) $\Rightarrow$ (ii). Since $x,\sigma^2(x),\dots,\sigma^k(x)$ are linearly dependent over $\f_q$, there exists $0\ne f=a_kX^k+a_{k-1}X^{k-1}+\cdots+a_2X^2+a_0\in\f_q[X]$ such that $(f(\sigma))(x)=0$. We claim that $a_k\ne0$ and $f\mid X^n-1$. Otherwise, $f_1:=\text{gcd}(f,X^n-1)$ has degree $<k$ and $(f_1(\sigma))(x)=0$. Then $x,\sigma(x),\dots,\sigma^{k-1}(x)$ are linearly dependent over $\f_q$, which is a contradiction. 
\end{proof}

\begin{cor}\label{C3.8}
If $X^n-1$ has a factor $X^k+a_{k-1}X^{k-1}+\cdots+a_2X^2+a_0\in\f_q[X]$, then $k\in\mathcal K_{n,q}$.
\end{cor}

\subsection{Numerical results}\

It follows from Propositions~\ref{1-2-order}, \ref{gcd} and Theorem~\ref{q-duality} that
\[
\mathcal K_{n,q}=\begin{cases}
\emptyset&\text{if}\ n=2,3,5,\cr
\{2\}&\text{if}\ n=4,\cr
\{2,3,4\}&\text{if}\ n=6.
\end{cases}
\]
For $q=3,5$ and $7\le n\le 11$, the answers to the question whether $g_{q-1}$ is $k$th order sum-free are tabulated in Table~\ref{tb1} (for $q=3$) and Table~\ref{tb2} (for $q=5$). These results are obtained through a combination of the theoretical results in \S4.5 and computer searches. It is remarkable that for $q=3,5$ and $n=7$, $g_{q-1}$ is $k$th order sum-free for all $1\le k\le 6$; such a phenomenon has not been observed in the binary case.

\begin{table}[h]
\caption{$q=3;\ 7\le n\le 11,\ 1\le k\le n-1$: Is $g_{q-1}$ $k$th order sum-free?}\label{tb1}
\vskip-0.2em
   \renewcommand*{\arraystretch}{1.3}
    \centering
     \begin{tabular}{c|cccccccccc}
         \hline
         $n\backslash k$  & 1& 2&3&4&5&6&7&8&9&10 \\  \hline 
         7& \cmark & \cmark  & \cmark  & \cmark  & \cmark  & \cmark\\ 
         8& \cmark &\xmark &\xmark  &\xmark &\xmark &\xmark &\cmark \\
         9& \cmark & \cmark &\xmark &\xmark  &\xmark &\xmark &\cmark &\cmark \\
         10& \cmark &\xmark &\xmark &\xmark &\xmark  &\xmark &\xmark &\xmark &\cmark \\
         11& \cmark & \cmark &\xmark &\xmark &\xmark &\xmark &\xmark &\xmark &\cmark &\cmark \\
         \hline
     \end{tabular}
\end{table}

\begin{table}[h]
\caption{$q=5;\ 7\le n\le 11,\ 1\le k\le n-1$: Is $g_{q-1}$ $k$th order sum-free?}\label{tb2}
\vskip-0.2em
   \renewcommand*{\arraystretch}{1.3}
    \centering
     \begin{tabular}{c|cccccccccc}
         \hline
         $n\backslash k$  & 1& 2&3&4&5&6&7&8&9&10 \\  \hline 
         7& \cmark & \cmark  & \cmark  & \cmark  & \cmark  & \cmark\\ 
         8& \cmark &\xmark &\xmark  &\xmark &\xmark &\xmark &\cmark \\
         9& \cmark & \cmark &\xmark &\xmark  &\xmark &\xmark &\cmark &\cmark \\
         10& \cmark &\xmark &\xmark &\xmark &\xmark  &\xmark &\xmark &\xmark &\cmark \\
         11& \cmark & \cmark &\xmark &? &\xmark &\xmark &?&\xmark &\cmark &\cmark \\
         \hline
     \end{tabular}
\end{table}

\section*{Appendix. Absolute Irreducibility of $\Theta_4$}
Recall that 
\[
\Theta_4(X_1,\dots,X_4)=m_{(2^3)}+m_{(2^2,2^2)}+m_{(2^2,2,2)}+m_{(2^2,2,1,1)}+m_{(2,2,2,2)},
\]
where
\[
m_{(2^3)}=X_1^{2^3}+X_2^{2^3}+X_3^{2^3}+X_4^{2^3},
\]
\[
m_{(2^2,2^2)}=X_1^{2^2}X_2^{2^2}+X_1^{2^2}X_3^{2^2}+X_1^{2^2}X_4^{2^2}+X_2^{2^2}X_3^{2^2}+X_2^{2^2}X_4^{2^2}+X_3^{2^2}X_4^{2^2},
\]
\begin{align*}
m_{(2^2,2,2)}=\,&X_2^{2^2}X_3^2X_4^2+X_2^2X_3^{2^2}X_4^2+X_2^2X_3^2X_4^{2^2}\cr
&+X_1^{2^2}X_3^2X_4^2+X_1^2X_3^{2^2}X_4^2+X_1^2X_3^2X_4^{2^2}\cr
&+X_1^{2^2}X_2^2X_4^2+X_1^2X_2^{2^2}X_4^2+X_1^2X_2^2X_4^{2^2}\cr
&+X_1^{2^2}X_2^2X_3^2+X_1^2X_2^{2^2}X_3^2+X_1^2X_2^2X_3^{2^2}.
\end{align*}
\begin{align*}
m_{(2^2,2,1,1)}=\,&X_1^{2^2}X_2^2X_3X_4+X_1^{2^2}X_2X_3^2X_4+X_1^{2^2}X_2X_3X_4^2\cr
&+X_1^2X_2^{2^2}X_3X_4+X_1X_2^{2^2}X_3^2X_4+X_1X_2^{2^2}X_3X_4^2\cr
&+X_1^2X_2X_3^{2^2}X_4+X_1X_2^2X_3^{2^2}X_4+X_1X_2X_3^{2^2}X_4^2\cr
&+X_1^2X_2X_3X_4^{2^2}+X_1X_2^2X_3X_4^{2^2}+X_1X_2X_3^2X_4^{2^2},
\end{align*}
\[
m_{(2,2,2,2)}=X_1^2X_2^2X_3^2X_4^2.
\]

\begin{taggedtheorem}{A1}\label{T-A1}
The polynomial $\Theta_4(X_1,\dots,X_4)$ is irreducible over $\overline\f_2$.
\end{taggedtheorem}

To prove Theorem~\ref{T-A1}, we need a couple of lemmas. We have 
\[
\Theta_4(X_1,\dots,X_4)=s_4^2+s_1(s_1s_2+s_3)s_4+(s_1^4+s_2^2+s_1s_3)^2=:G(s_1,\dots,s_4),
\]
where $s_i$ is the $i$th elementary symmetric polynomials in $X_1,\dots,X_4$. 

\begin{taggedlemma}{A2}\label{L-A2}
The polynomial $G(s_1,\dots,s_4)$ is irreducible in $\overline\f_2[s_1,\dots,s_4]$.
\end{taggedlemma}

\begin{proof}
We treat $G(s_1,\dots,s_4)$ as a polynomial in $s_4$.  Assume to the contrary that
\[
G=(s_4+A)(s_4+B)
\]
for some $A,B\in\overline\f_2[s_1,s_2,s_3]$. Then
\[
\begin{cases}
A+B=s_1(s_1s_2+s_3),\cr
AB=(s_1^4+s_2^2+s_1s_3)^2.
\end{cases}
\]
Clearly, $s_1^4+s_2^2+s_1s_3$ is irreducible in $\overline\f_2[s_1,s_2,s_3]$. Hence $A=a(s_1^4+s_2^2+s_1s_3)^i$ and $B=b(s_1^4+s_2^2+s_1s_3)^{2-i}$ for some $a,b\in\overline\f_2^*$ and $i\in\{0,1,2\}$. In any case, $A+B$ does not contain any monomial of degree $3$. Thus $A+B\ne s_1(s_1s_2+s_3)$, which is a contradiction. 
\end{proof}

Lemma~\ref{L-A2} implies that $\Theta_4(X_1,\dots,X_4)$ does not have a proper factor that is symmetric in $X_1,\dots,X_4$.

We have
\[
\Theta_4(X_1,X_2,X_3,0)=H(X_1,X_2,X_3)^2,
\]
where
\[
H(X_1,X_2,X_3)=X_1^4+X_2^4+X_3^4+X_1^2X_2^2+X_1^2X_3^2+X_2^2X_3^2+X_1X_2X_3(X_1+X_2+X_3).
\]

\begin{taggedlemma}{A3}\label{L-A3}
The polynomial $H(X_1,X_2,X_3)$ is irreducible over $\overline\f_2$.
\end{taggedlemma}

\begin{proof}
We treat $H$ as a polynomial in $X_3$:
\[
H=X_3^4+(X_1^2+X_2^2+X_1X_2)X_3^2+X_1X_2(X_1+X_2)X_3+(X_1^2+X_2^2+X_1X_2)^2,
\]
where
\[
X_1^2+X_2^2+X_1X_2=(X_1+\alpha X_2)(X_1+\alpha^{-1}X_2),\quad \alpha\in\f_{2^2}\setminus\f_2.
\]

First assume that $H$ has a factor $X_3+A$, where $A\in\overline\f_2[X_1,X_2]$ is homogeneous of degree $1$. Then $A\mid (X_1+\alpha X_2)(X_1+\alpha^{-1}X_2)$, say $A=c(X_1+\alpha X_2)$, where $c\in\overline \f_2^*$. Then we have
\begin{align*}
0=\,&H(X_1,X_2,c(X_1+\alpha X_2))\cr
=\,&(1+c+c^2)^2X_1^4+c(1+c)X_1^3X_2+(1+c)(1+\alpha c)X_1^2X_2^2\cr
&+c(\alpha+c+\alpha c)X_1X_2^3+(1+c)^2(1+\alpha c^2)X_2^4.
\end{align*}
It follows that
\[
\begin{cases}
1+c+c^2=0,\cr
c(1+c)=0,
\end{cases}
\]
which is impossible.

\medskip

Now assume that $H$ has an absolute irreducible factor $h=X_3^2+AX_3+B$, where $A,B\in\overline\f_2[X_1,X_2]$. It follows that $B=c(X_1+\alpha X_2)^2$ or $c(X_1^2+X_2^2+X_1X_2)$, where $c\in\overline\f_2^*$.

\medskip
{\bf Case 1.} Assume that $B=c(X_1+\alpha X_2)^2=c(X_1^2+\alpha^{-1}X_2^2)$. Then $X_3^2+A(X_2,X_1)X_3+B(X_2,X_1)$ is a {\em different} irreducible factor of $H$ (different since $B(X_2,X_1)\ne B(X_1,X_2)$). Thus
\[
H=(X_3^2+A(X_1,X_2)X_3+B(X_1,X_2))(X_3^2+A(X_2,X_1)X_3+B(X_2,X_1)).
\]
Comparing the coefficients gives
\begin{align*}
&A(X_1,X_2)+A(X_2,X_1)=0,\cr
&B(X_1,X_2)+B(X_2,X_1)+A(X_1,X_2)A(X_2,X_1)=X_1^2+X_2^2+X_1X_2.
\end{align*}
It follows that
\[
c\alpha(X_1^2+X_2^2)+A(X_1,X_2)^2=X_1^2+X_2^2+X_1X_2.
\]
This is impossible since $X_1X_2$ does not appear at the left hand side.

\medskip
{\bf Case 2.} Assume that $B=c(X_1^2+X_2^2+X_1X_2)$. 

\medskip
{\bf Case 2.1.} Assume that $A(X_1,X_2)=A(X_2,X_1)$. Then $A(X_1,X_2)=b(X_1+X_2)$ for some $b\in\overline\f_2$, and hence
\[
h=X_3^2+b(X_1+X_2)X_3+c(X_1^2+X_2^2+X_1X_2).
\]
However,
\begin{align*}
H\equiv\,&
X_3 \left(b^3 X_1^3+b^3 X_1^2 X_2+b^3 X_1 X_2^2+b^3
   X_2^3+b X_1^3+b X_2^3+X_1^2 X_2+X_1
   X_2^2\right)\cr
   &+b^2 c X_1^4+b^2 c X_1^3 X_2+b^2 c X_1
   X_2^3+b^2 c X_2^4+c^2 X_1^4+c^2 X_1^2 X_2^2+c^2
   X_2^4\cr
   &+c X_1^4+c X_1^2 X_2^2+c
   X_2^4+X_1^4+X_1^2 X_2^2+X_2^4\cr
\not\equiv\,&0\pmod h,   
\end{align*}
which is a contradiction.

\medskip
{\bf Case 2.2.} Assume that $A(X_1,X_2)\ne A(X_2,X_1)$. Then
\[
h(X_2,X_1,X_3)=X_3^2+A(X_2,X_1)X_3+B(X_2,X_1)
\]
is a different irreducible factor of $H$. Thus
\[
H=(X_3^2+A(X_1,X_2)X_3+B(X_1,X_2))(X_3^2+A(X_2,X_1)X_3+B(X_2,X_1)).
\]
Comparing the coefficients of $X_3^3$ gives $0=A(X_1,X_2)+A(X_2,X_1)$, which is a contradiction.
\end{proof}

\begin{proof}[Proof of Theorem~\ref{T-A1}]
Assume to the contrary that $\Theta_4$ has an absolutely irreducible factor $f$ of degree $d$ with $0<d<8$. Treating $f$ as a polynomial in $X_4$ over $\overline\f_2[X_1,X_2,X_3]$, we may write 
\[
f=X_4^d+A_{d-1}X_4^{d-1}+\cdots+A_0,
\]
where $A_i\in\overline\f_2[X_1,X_2,X_3]$ is homogeneous of degree $d-i$. Then $A_0\mid \Theta_4(X_1,X_2,X_3,0)=H^2$. Since $H$ is absolutely irreducible (Lemma~\ref{L-A3}) and $\deg H=4$, we must have $d=4$ and $A_0=cH$ for some $c\in\overline\f_2$, so 
\[
f=X_4^4+A_3X_4^3+A_2X_4^2+A_1X_4+cH.
\]
The symmetric group $\frak S_3$ acts on $\overline\f_2[X_1,\cdots,X_4]$ by permuting $X_1,X_2,X_3$. We claim that $f$ is symmetric in $X_1,X_2,X_3$. Otherwise, there exists $\sigma\in\frak S_3$ such that
\[
\sigma(f)=X_4^4+\sigma(A_3)X_4^3+\sigma(A_2)X_4^2+\sigma(A_1)X_4+cH
\]
is a different irreducible factor of $\Theta_4$. It follows that
\[
\Theta_4=f\sigma(f)=X_4^8+\cdots+(A_1+\sigma(A_1))cHX_4+c^2H^2.
\]
However, the coefficient of $X_4$ in $\Theta_4$ is
\[
X_1X_2X_3(X_1+X_2)(X_1+X_3)(X_2+X_3)(X_1+X_2+X_3),
\]
which is not divisible by $H$. This is a contradiction, hence the claim is proved.

We may also treat $f$ as a polynomial in $X_1$ over $\overline\f_2[X_2,X_3,X_4]$. By the same argument, $f$ is also symmetric in $X_2,X_3,X_4$. It follows that $f$ is symmetric in $X_1,\dots,X_4$. However, this is impossible by Lemma~\ref{L-A2}.

The proof of Theorem~\ref{T-A1} is now complete.
\end{proof}
  
We suspect that in general, $\Theta_k$ is absolutely irreducible for $k\ge 3$, although the proof for the case $k=4$ given here is already quite technical. 


\section*{Acknowledgments}

This work was supported by NSF REU Grant 2244488.




\end{document}